\documentclass[reqno]{amsart}
\usepackage[utf8]{inputenc}
\usepackage[english]{babel}
\usepackage{amsmath}
\usepackage{amsthm}
\usepackage{amssymb}
\usepackage{mathtools}
\usepackage{enumitem}
\usepackage{xcolor}
\usepackage[centering]{geometry}
\usepackage{pgfplots}
\usepackage{tikz}
\usepackage{tikzrput}
\usetikzlibrary{calc}
\usetikzlibrary{positioning}
\usetikzlibrary{decorations,decorations.markings}
\usetikzlibrary{decorations.pathreplacing,shapes.misc}
\usetikzlibrary{arrows}

\usepackage{fp}  		
\usepackage{pgfplots}
\usetikzlibrary{math} 
\usetikzlibrary{angles,quotes}
\definecolor{colQ}{rgb}{0.8, 0.8, 0.8} 

\usepackage{fancyhdr}
\usepackage{soul}  

\numberwithin{equation}{section}

\usepackage[hypertexnames=false, breaklinks,
   colorlinks=true, 
   citecolor=red!70!black,
   linkcolor=blue,
   anchorcolor=cyan,
   urlcolor=green!50!black,
   pagebackref=false,
   ]{hyperref}

\colorlet{mycyan}{cyan!20}
\colorlet{mymagenta}{magenta!30}
\colorlet{myorange}{orange!40}

\usepackage[color=yellow!30, linecolor=orange, colorinlistoftodos]{todonotes}
\usepackage[breakable]{tcolorbox}
\tcbuselibrary{theorems}
\newtcolorbox{bluebox}[1][]%
{left=0mm, right=0mm, bottom=0mm, top=0mm, sharp corners, boxrule=.8pt, before skip=\topsep, after skip=\topsep, colback=cyan!5, colframe=cyan, coltitle=black, fonttitle=\bfseries, title=#1, breakable}
\tcbset{tcbox raise base, sharp corners, left=0mm,right=0mm,top=0mm,bottom=0mm,nobeforeafter, boxrule =.5pt} 
\tcbset{tcbox raise base, sharp corners, left=0mm,right=0mm,top=0mm,bottom=0mm,nobeforeafter, boxrule =.5pt}

\theoremstyle{definition}
\newtheorem{theorem}{Theorem}[section]
\newtheorem{defi}[theorem]{Definition}
\newtheorem{exam}[theorem]{Example}
\newtheorem{remark}[theorem]{Remark}
\newtheorem*{remark*}{Remark}

\newtheorem{lemma}[theorem]{Lemma}
\usepackage{mathrsfs}
\newtheorem{corollary}[theorem]{Corollary}

\newcommand{\C}{\mathbb{C}}
\newcommand{\N}{\mathbb{N}}
\newcommand{\R}{\mathbb{R}}
\newcommand{\Z}{\mathbb{Z}}

\newcommand{\snull}{\sigma}

\DeclareMathOperator{\ACloc}{AC_{loc}}
\DeclareMathOperator{\re}{Re}
\DeclareMathOperator{\im}{Im}

\DeclareMathOperator{\cco}{\overline{co}}
\DeclareMathOperator{\interior}{Int}
\newcommand{\rd}{\mathrm d} 
\def\PT{{\mathcal P}{\mathcal T}}
\def\cP{{\mathcal P}}      
   \def\cT{{\mathcal T}}

\newcommand{\define}[1]{\textbf{#1}}
\pgfplotsset{compat=newest}

\begin{document}

\title[Limit point/limit circle for Sturm-Liouville problems with complex potentials]{Limit point and limit circle trichotomy for Sturm-Liouville problems with complex potentials}

\author[F.\ Leben]{Florian Leben}
\address{Institut f\"{u}r Mathematik\\
Technische Universit\"at Ilmenau\\ Postfach 100565\\
98648 Ilmenau\\ Germany}
\email{leben.florian@gmx.de}

\author[E.\ Leguizam\'on]{Edison Leguizam\'on}
\address{Departamento de Matem\'aticas\\
Universidad de Los Andes\\
Cra.\ 1a No 18A-70\\
111711 Bogotá\\
Colombia}
\email{ej.leguizamon@uniandes.edu.co}
\urladdr{\url{https://academia.uniandes.edu.co/AcademyCv/ej.leguizamon?isPhdStudent=true&hideNavBar=false}}

\author[C.\ Trunk]{Carsten Trunk}
\address{Institut f\"{u}r Mathematik\\
Technische Universit\"at Ilmenau\\ Postfach 100565\\
98648 Ilmenau\\ Germany}
\email{carsten.trunk@tu-ilmenau.de}
\urladdr{\url{https://www.tu-ilmenau.de/de/analysis/team/carsten-trunk/}}

\author[M.\ Winklmeier]{Monika Winklmeier}
\address{Departamento de Matem\'aticas\\
Universidad de Los Andes\\
Cra.\ 1a No 18A-70\\
111711 Bogotá\\
Colombia}
\email{mwinklme@uniandes.edu.co}
\urladdr{\url{https://math.uniandes.edu.co/~mwinklme/index.php}}

\date{\today}

\maketitle
\begin{abstract}
The limit point and limit circle classification of real Sturm-Liouville problems by H.~Weyl more than 100 years ago was extended by A.R.\ Sims around 60 years ago to the case when the coefficients are complex.
Here, the main result is a collection of various criteria which allow us to decide to which class of Sims' scheme a given Sturm-Liouville problem with complex coefficients belongs.
This is subsequently applied to a second order differential equation
defined on a ray in $\mathbb C$ which is motivated by the recent intensive research connected with $\mathcal P \mathcal T$-symmetric Hamiltonians.
\end{abstract}

\section{Introduction}

The search for criteria that guarantee that a Sturm-Liouville equation with real coefficients is in the limit point or the limit circle case  has a long tradition since the seminal paper of  H.~Weyl \cite{weyl} in 1910
where he classifies Sturm-Liouville problems on an interval $(a,b)$ with a regular end point $a$ into two classes:
Either all solutions of the eigenvalue problem are square-integrable close to $b$ (limit circle) 
or there exists at least one solution without this property (limit point).
This behaviour is independent of the chosen eigenvalue parameter.


Sturm-Liouville
problems with \textit{complex-valued} coefficients were investigated 
in the (also seminal) paper by A.R.~Sims in 1957 \cite{SIMS}
with further refinements in
\cite{MC} and \cite{QiSun,QiSunZheng}. The classification 
proposed by A.R.~Sims and later Brown et al. in \cite{MC}
contains three different cases, where 
one takes into account also the behaviour of the solution and its derivative in a weighted $L^2$-space (for details we refer to Section \ref{Samarkand}).

Our interest in the classification proposed by A.R.~Sims
arises from a second order differential equation defined on a 
ray in $\mathbb C$.  This is motivated
by recent intensive research connected with $\mathcal P \mathcal T$-symmetric
Hamiltonians, cf.\ \cite{bender}.
In the seminal paper by C.M.~Bender
and S.~Boettcher \cite{bender} a new view of quantum mechanics was proposed
which adopts all its axioms except the one that restricts the Hamiltonian to be Hermitian,  relaxing it to the assumption that the Hamiltonian is $\cP\cT$-symmetric. Here, $\mathcal P$ is parity and $\mathcal T$ is time reversal.
 Since then, $\cP\cT$-symmetric Hamiltonians have been intensively analysed by many authors.
In \cite{M02} $\cP\cT$-symmetry was embedded into a more general mathematical framework: pseudo-Hermiticity or, equivalently,
self-adjoint operators in  Krein spaces, \cite{AT12,HK,LT04,LebTr}.
For a general introduction to $\cP\cT$-symmetric quantum mechanics we refer
to \cite{M10} and \cite{BenderBook}.
\smallskip

Let $\Gamma$ be a ray in the complex plane with angle $\phi \in (-\pi/2,\pi/2)$,
\begin{equation*}
\Gamma:=\{z\in\C:z=xe^{i\phi}, x\in [a,\infty)\}
\end{equation*}
for some $a\geq 0$.
Our main interest is to obtain a Weyl criterion for the differential equation
\begin{equation}
\label{ equation 1-gamma}
-y(z)''+ \textbf{q}(z)y(z)=\lambda y(z),
\qquad z\in \Gamma,
\end{equation}
where 
$\textbf{q}:\Gamma \to \C$ is locally integrable.

A prominent class of potentials consists of the $\PT$-symmetric potentials
$$
\textbf{q}(z):= - (iz)^{N+2}
$$
where $N$ is a positive integer \cite{bender,BBJ02}.
Other Hamilitonians can be found for instance in \cite{CCS08,M10}.

Via a parametrization  \eqref{ equation 1-gamma} can be mapped back to the real line leading to a Sturm-Liouville problem on an interval of the form
\begin{equation}\label{weltmeister}
    -\big( py' \big)'+qy=w\lambda y.
\end{equation}
We give an asymptotic approximation for its solutions in Section~\ref{Tash}
via an approach based on \cite{east}.
A careful analysis of these asymptotic approximations leads
to new criteria for limit point/limit circle cases in the sense of
A.R.~Sims for the equation in \eqref{weltmeister} 
(see Section \ref{usbek})
and then,
by parameterisation, also for \eqref{ equation 1-gamma},
see Section \ref{Russisch}.

\medskip

\textbf{Notations.}
For $-\infty < a < b \le \infty$ we denote by $\ACloc(a,b)$ the set of locally absolutely continuous functions on each compact subinterval of $(a,b)$.
For a locally integrable function $w:(a,b)\to\C$ we set
$L_w^2(a,b) := \{ f: (a,b) \to \C : f \text{ measurable},\ \int_a^b |f(x)|^2 |w(x)|\,\rd x < \infty \}$.
If $w = 1$, then we write $L^2(a,b)$.
Recall that the normed space of uniformly locally integrable functions $  L^1_{\mathrm u}(a,b)$ is defined as
\begin{align*}
    L^1_{\mathrm u}(a,b)=\left\{f\in L^1_{\mathrm{loc}}(a,b) : \sup_{n\in\mathbb Z} \int_{[n,n+1]\cap(a,b)} |f(t)|\,\rd t <\infty\right\}.
\end{align*}


\section{Weyl's alternative for complex Sturm-Liouville problems} \label{Samarkand}
Consider the Sturm-Liouville problem
\begin{equation}
\label{1}
    -\big( p(x)y'(x) \big)'+q(x)y(x)=w(x)\lambda y(x),\quad x\in [a,b)
\end{equation}
where $a\in \R$, $b\in \R\cup \{\infty\},$ $\lambda\in \C$ 
and $w,\, 1/p,\, q:[a,b) \to \C$ are locally integrable in $[a,b)$ and satisfy
$w(x)>0$, $p(x)\neq 0$  for a.a.\ $x\in[a,b)$. 
Here we always assume that the end point $a$ is regular and $b$ is singular (which is indicated by writing $x\in [a,b)$
or $L^{1}_{\textnormal{loc}}[a,b)$).
A \define{solution} for \eqref{1} is a function $y$ such that
$y,\, py'\in\ACloc(a,b)$ and $y$ satisfies \eqref{1} for a.a.\ $x\in[a,b)$.
\smallskip

Given $A\subset \C$, we define $\cco(A)$ as the closed convex hull of $A$.
We impose
\begin{equation}
\label{1.2}
    Q:=\cco\left\{\frac{q(x)}{w(x)}+rp(x)\ :\ 0<r<\infty,\ x\in[a,b)\right\}\neq \C.
\end{equation}
Let $\lambda \notin Q$.
Then there exists a unique point $K\in Q$ which minimizes the distance between $\lambda$  and $Q$.
Moreover, there exists an angle $\theta$ such that
\begin{equation}
   \label{condition}
   \re [e^{i\theta}\left(z-K\right)]\geq 0\quad \mbox{for all }z\in Q.
\end{equation}
In fact, let  $Q-K:=\{z\in\C:z=w-K, w\in Q\}$.
Since the set $\interior(Q-K) = \interior(Q)-K$ is open and convex and does not intersect the trivial subspace $\{0\}$, the geometric form of the Hahn-Banach Theorem
(see e.g. \cite[Theorem~7.7.4]{naricibeckenstein})
shows that there is a linear functional
$f:\C\to\C$ such that $\re[f(z)]>0$ for all $z\in\interior(Q)-K$.
Without restriction we may assume that $f$ is normalized.
Hence there exists a  real number $\theta$ such that
$f(z)=e^{i\theta}z$ for all $z\in\C$ and $\re[e^{i\theta}v]>0$ for all $v\in \interior(Q)-K$.
\smallskip

For $K\in\C$ and $\theta\in\R$ we define the open half-plane
\[\
\Lambda_{K,\theta}:=\left\{z\in \C: \re [e^{i\theta}\left(z-K\right)]< 0\right\},
\]
and the set
\[S:=\left\{(\theta,K):\eqref{condition} \mbox{ is satisfied}\right\}.\]

In \cite{MC} it is proved that the Sturm-Liouville problem \eqref{1} falls in exactly one of the cases of the next definition.

\begin{defi}
   \label{d1.1}
   Given $(\theta, K)\in S$ and $\lambda\in \Lambda_{K,\theta},$ we have the following cases:
   \begin{enumerate}[label=\upshape(\arabic*)]
      \item\label{limitpointI}
      There is, up to a multiplicative constant, only one solution $y$ of the equation \eqref{1} such that
      \begin{equation}
	 \label{1.3}
	 \int_{a}^{b}\re(e^{i\theta}p)|y'|^{2}\,\rd t+\int_{a}^{b}\re[e^{i\theta}(q-Kw)]|y|^{2}\, \rd t+\int_{a}^{b}w|y|^{2}\, \rd t<\infty
      \end{equation}
      and this is the only solution belonging to
      $L_{w}^{2}(a,b)$.
      In this case we say that \eqref{1} is in the \textbf{limit point I case}.

      \item\label{limitpointII}
      All solutions of \eqref{1} are in $L^{2}_{w}(a,b)$ but, up to a multiplicative constant, there is only one solution that satisfies \eqref{1.3}.
      In this case we say that \eqref{1} is in the \textbf{limit point II case}.

      \item\label{limitcircle} All solutions of \eqref{1} are in $L^{2}_{w}(a,b)$ and all solutions satisfy \eqref{1.3}.
      In this case we say that \eqref{1} is in the \textbf{limit circle case}.
        \end{enumerate}
\end{defi}
\begin{remark}\label{CanNeverLove}
      In the situation of Definition \ref{d1.1} (3) we have
      \begin{align*}
          \re[ e^{i\theta}(q-Kw) ]
          = \re\left[ w e^{i\theta}\left(\frac{q}{w}-K\right) \right]
          = w\re\left[ e^{i\theta}\left(\frac{q}{w}-K\right) \right]
          \ge 0.
      \end{align*}
  Hence the three summands on the left hand side of \eqref{1.3} are nonnegative and therefore, if a solution of \eqref{1} satisfies \eqref{1.3},
  then it is automatically in $L^{2}_{w}(a,b)$.
   \end{remark}

\begin{remark}\label{CanNeverLoveAgain}
   In \cite[Remark 2.2]{MC} the method of variation of parameters is used to deduce that the classification is independent of $\lambda$, that is
   \begin{itemize}
      \item If all solutions are in $L^{2}_w(a,b)$ for some $\lambda_{0}\in \C$, then the same is true for all $\lambda\in \C.$
      \item If all solutions satisfy \eqref{1.3} for some $\lambda_0\in \Lambda_{\theta,K}$, then the same is true for all $\lambda\in \C.$
   \end{itemize}
\end{remark}

Note that $\theta=\frac{\pi}{2}$ in the case of real coefficients.
Hence the first two terms in \eqref{1.3} are zero and therefore the limit point case II is not possible for real Sturm-Liouville equations.


\section{Asymptotic approximation of second order differential equations} \label{Tash}
   
In this section we derive an asymptotic approximation of solutions of \eqref{2.1}.
Since \eqref{ equation 1-gamma} can be transformed into such an equation, this will allow us to establish limit point/limit circle criteria in Section~\ref{Russisch}.
Our approximations of the solutions are primarily based on \cite[Theorem~1.3.1]{east}.
Consider the differential equation
\begin{equation}
   \label{2.1}
   (py')'(x)=s(x)y(x),\quad x\in [a,b)
\end{equation}
where $s,p:[a,b)\to \C$ are
 functions such that
$\frac{1}{p},s\in L^{1}_{\textnormal{loc}}[a,b)$.
In what follows we define the $n$th root of a complex number $z=re^{i\arg(z)}$ with $-\pi<\arg (z)\leq\pi$ as $z^{1/n}=r^{1/n}e^{i\arg(z)/n}$. Our first theorem is a variation of Theorem 2.5.1 in  \cite{east}. It leads to slightly different assertions which form the basis for the subsequent section.

\begin{theorem}\label{t2.2} 
Assume that $p(x)\neq0$, $s(x)\neq0$, $\arg \frac{s(x)}{p(x)} \neq \pi$, $\arg p(x)s(x)\neq\pi$ for all $x\in[a,b)$ and 
let $u:=(ps)^{-1/4}$.
Assume that $u, pu' \in \ACloc(a,b)$
and
$u(pu')'\in L^{1}(a,b)$. Then there exist $F_{j},\widehat{F}_{j}:[a,b)\to \C$ and a fundamental system $\{y,\widehat{y}\}$ for \eqref{2.1} with
\begin{alignat}{3}
\label{y}
y(x) &= (p(x)s(x))^{-1/4}e^{-\int_{a}^{x}\sqrt{s(t)/p(t)}\, \rd t}(F_{1}(x)+1),
&& \quad x\in[a,b),
\\
\nonumber
((ps)^{1/4}y)'(x) &= (s(x)/p(x))^{1/2}e^{-\int_{a}^{x}\sqrt{s(t)/p(t)}\, \rd t}(F_{2}(x)-1),
&& \quad x\in[a,b),
\\[2ex]
\label{widey}
\widehat{y}(x) &= (p(x)s(x))^{-1/4}e^{\int_{a}^{x}\sqrt{s(t)/p(t)}\, \rd t}(\widehat{F}_{1}(x)+1),
&& \quad x\in[a,b),
\\
\nonumber
((ps)^{1/4}\widehat{y})'(x) &= (s(x)/p(x))^{1/2} e^{\int_{a}^{x}\sqrt{s(t)/p(t)}\, \rd t}(\widehat{F}_{2}(x)  - 1),
&& \quad x\in[a,b),
    \end{alignat}
    and, for $j=1,2,$
    \begin{equation*}
\|F_{j}\|_{\infty},\ \|\widehat{F}_{j}\|_{\infty}\leq 2e^{2M}-2,
\quad\text{and}\quad
F_{j}(x),\ \widehat{F}_{j}(x)\to 0,\quad\mbox{when }x\to b,
    \end{equation*}
    where $\|\cdot\|_{\infty}$ is the supremum norm and $M:=\|u(pu')'\|_{L^{1}}$.
\end{theorem}
\begin{proof}
     Note that $y$ is a solution for $\eqref{2.1}$ if and only if $Y$ is a solution for $Y'(x)=A(x)Y(x)$ where
     \begin{equation}
        \label{y-sistem}
        Y(x):=\begin{pmatrix} y(x) \\ p(x)y'(x)
        \end{pmatrix}
        \quad\mbox{ and }\quad
        A(x):=\begin{pmatrix}
           0 & 1/p(x) \\
           s(x)& 0
        \end{pmatrix} \quad x\in[a,b).
    \end{equation}
    We call $Y$ a solution for $Y'=AY$ if $Y$ is locally absolutely continuous in $[a,b)$ and satisfies the differential equation \eqref{y-sistem} for a.a.\ $x\in[a,b)$.

    For $z(x):=u(x)^{-1}y(x)$ and
    \begin{equation}
        \label{Z}
        Z(x):=\begin{pmatrix} z(x) \\ p(x)u(x)^{2}z'(x)
        \end{pmatrix}
    \end{equation}
    we obtain the relations
    \begin{equation}
        \label{z-system}
        Y(x)
        = \begin{pmatrix}
        u(x) &0  \\
        p(x)u'(x) & u^{-1}(x)
	\end{pmatrix}Z(x)
    \end{equation}
    and
    \[Z(x)=\begin{pmatrix}
    u^{-1}(x)&0  \\
    -p(x)u'(x)& u(x)
    \end{pmatrix}Y(x).\]
    The transformation \eqref{z-system} takes \eqref{y-sistem} into
      \begin{equation}
	 \label{z-system2}
	 Z'(x)=
	 \begin{pmatrix}
	    0 & p(x)^{-1}u(x)^{-2} \\
	    s(x)u(x)^{2}-u(x)\left(p(x)u'(x)\right)' & 0
	 \end{pmatrix}
	 Z(x), \qquad x\in [a,b).
      \end{equation}
      By the transformation
      \begin{equation}
	 \label{W}
	 W(x):=\frac{1}{2}\begin{pmatrix*}[r]
	    1& 1 \\ 1&-1
	 \end{pmatrix*}Z(x)
      \end{equation}
      with inverse transformation
      \[Z(x)=\begin{pmatrix*}[r]
	 1 & 1  \\ 1 & -1
      \end{pmatrix*}W(x)\]
we obtain the following system, using that 
$p^{-1}u^{-2}=su^{2}=\sqrt{s/p}\in L^{1}_{\textnormal{loc}}[a,b)$,  
      \begin{equation}
	 \label{eq:Wequation}
	 W'(x)=
	 \begin{pmatrix}
	    \left(\frac{s(x)}{p(x)}\right)^{1/2} & 0 \\
	    0&-\left(\frac{s(x)}{p(x)}\right)^{1/2}
	 \end{pmatrix}W(x)
	 -\frac{1}{2}(u(x)(pu')'(x))
	 \begin{pmatrix*}[r]
	    1 & 1 \\ -1 & -1
	 \end{pmatrix*}W(x).
      \end{equation}
      If we set
      \[S(x):=\frac{-u(x)(pu')'(x)}{2}
      \begin{pmatrix*}[r]
	 1& 1 \\
	 -1& -1
      \end{pmatrix*},\]
      then \eqref{eq:Wequation} can be written as
      \begin{equation}
	 \label{eq:Wequation2}
	 W'(x)=
	 \left( \frac{s(x)}{p(x)}\right)^{1/2}
	 \begin{pmatrix}
	    1  & 0 \\ 0 & -1
	 \end{pmatrix}W(x)
	 + S(x) W(x).
      \end{equation}
      Note that by hypothesis $\int_{a}^{b}\|S(t)\|_{\C^2}\, \rd t = \int_{a}^{b}\|u(pu')'(t)\|\, \rd t=M<\infty$ and
      \begin{equation*}
	 M(x) := \int_{x}^{b}\|S(t)\|_{\C^2}\, \rd t \le  M<\infty,
	 \qquad x\in [a,b)
      \end{equation*}
 where $\|\cdot\|_{\C^2}$ denotes the operator norm of a $2\times 2$ matrix.
      Now we will construct two linearly independent solutions of \eqref{eq:Wequation} using a fixed point argument.
      For our first solution, we set
      $V(x):=W(x)e^{\int_{a}^{x}\sqrt{s(t)/p(t)}\, \rd t}$ which solves the differential equation
      \begin{equation}
	 \label{2.2}
	 V'(x)=\begin{pmatrix}
	    2\left(\frac{s(x)}{p(x)}\right)^{1/2} &0  \\
	    0& 0
	 \end{pmatrix}V(x)
	 + S(x) V(x),
	 \qquad x\in [a,b).
      \end{equation}
      A fundamental system for the homogeneous differential equation
      \begin{equation}
	 \label{fund}
	 V_0'(x)=\begin{pmatrix}
	    2\left(\frac{s(x)}{p(x)}\right)^{1/2} &0  \\
	    0&0
	 \end{pmatrix}V_0(x)
      \end{equation}
      is given by the matrix
      \[\Phi(x)=
      \begin{pmatrix}
	 e^{2\int_{a}^{x}\sqrt{s(t)/p(t)}\, \rd t} & 0  \\
	 0& 1
      \end{pmatrix}.\]
      Since $\re \left[ p(x)s(x) \right]^{1/2}\geq 0$ in $[a,b)$, we obtain
      \begin{equation}
	 \label{norma1}
	 \|\Phi(x)\Phi(\widehat{x})^{-1}\|_{2}
	 =
	 \left\|\begin{pmatrix}
	    e^{-2\int_{x}^{\widehat{x}}\left(\frac{s(t)}{p(t)}\right)^{1/2}\, \rd t} & 0  \\
	    0& 1
	 \end{pmatrix}\right\|_{\C^2}=1,
	 \qquad a\leq x<\widehat{x}.
      \end{equation}
      Let $C([a,b),\C^2)$ denote the Banach space of continuous bounded functions $f:[a,b)\to \C^{2}$ equipped with the supremum norm $\|\cdot\|_{\infty}$.
      We define the continuous operator $F:C([a,b),\C^2)\to C([a,b),\C^2)$ by
      \begin{equation}
	 \label{F}
	 (Ff)(x):=
	 \begin{pmatrix} 0  \\ 1
	 \end{pmatrix}
	 -\Phi(x)\int_{x}^{b}\Phi(t)^{-1}S(t)f(t)\, \rd t.
      \end{equation}
      Next we construct a sequence of functions $\{h_{k}\}_{k\in\N}$ by
      \begin{equation*}
	 h_{1}:=\begin{pmatrix} 0  \\ 1
	 \end{pmatrix},
	 \qquad
	 h_{k+1}(x)=(Fh_{k})(x),
	 \quad k\ge 1.
      \end{equation*}
      We proceed by induction to  prove that for all $x\ge a$
      \begin{equation}
	 \label{eq:pointwise}
	 \|h_{k+1}(x)-h_{k}(x)\|
	 \leq \frac{1}{k!}\left( \int_x^b \|S(t)\|_{\C^2}\,\rd t \right)^{k}
	 =  \frac{(M(x))^{k}}{k!}
	 \le  \frac{M^{k}}{k!}.
      \end{equation}
      To this end, note that for all $x\geq a$
      \begin{align*}
	 \|h_{2}(x)-h_{1}(x)\|
	 &= \left\|\, -\Phi(x)\int_{x}^{b}\Phi(t)^{-1}S(t)h_{1}(t)\, \rd t \, \right\|
	 \leq \int_{x}^{b}\|S(t)\|_{\C^2}\, \rd t
	 = M(x) \leq M.
      \end{align*}
      By the induction hypothesis we obtain
      \begin{align*}
	 |h_{k+1}(x)-h_{k}(x)|
	 &\leq \int_{x}^{b} \|S(t)\|_{\C^2}|h_{k}(t)-h_{k-1}(t)| \, \rd t\notag\\
	 &\leq \int_{x}^{b} \|S(t)\|_{\C^2}
	 \frac{\left(\int_{t}^{b}\|S(u)\|_{\C^2}\, \rd u\right)^{k-1}}{(k-1)!}\, \rd t \notag\\
	 &\leq\left|-\int_{x}^{b}
	 \frac{\rd }{\rd t} \frac{ \left(\int_{t}^{b}\|S(u)\|_{\C^2}\, \rd u\right)^{k} }{ k! }\, \rd t
	 \right|\notag\\
	 &= \frac{1}{k!}\left(\int_{x}^{b}\|S(u)\|_{\C^2}\, \rd u\right)^{k}
	 = \frac{(M(x))^{k}}{k!}
	 \leq \frac{M^{k}}{k!}.
      \end{align*}
      For all $n\geq m\geq 1$ we have that
      \begin{equation*}
	 \|h_{n}(x)-h_{m}(x)\|
	 = \left\|\sum_{k=m}^{n-1}h_{k+1}(x)-h_{k}(x)\right\|
	 \leq \sum_{k=m}^{n-1}\frac{(M(x))^{k}}{k!}\leq e^{M(x)}
      \end{equation*}
      and consequently
      $\|h_{n} - h_{m}\|_\infty \leq e^M$,
      so $\{h_{n}\}_{k\in\N}$ converges uniformly to a bounded and continuous function $h\in C([a,b),\C^2)$.
      Moreover, $h$ is a fixed point of $F$ because $F$ is continuous.
      Let us show that $h$ is a solution of \eqref{2.2}.
      \begin{align*}
	 h'(x)&= \frac{\rd}{\rd x}(F(h))(x)\\
	 &= \begin{pmatrix}
	    2\sqrt{s(x)/p(x)} & 0  \\
	    0 & 0
	 \end{pmatrix}h_{1}
	 -\Phi'(x)\int_{x}^{b}\Phi(t)^{-1}S(t)h(t)\, \rd t+S(x)h(x)\\
	 &= \begin{pmatrix}
	    2\sqrt{s(x)/p(x)}&0  \\
	    0& 0
	 \end{pmatrix}
	 \left(h_{1}-\Phi(x)\int_{x}^{b}\Phi(t)^{-1}S(t)h(t)\, \rd t\right)+S(x)h(x)
	 \\
	 &= \begin{pmatrix}
	    2\sqrt{s(x)/p(x)} & 0  \\
	    0 & 0
	 \end{pmatrix}F(h)(x)
	 +S(x)h(x)
	 = \begin{pmatrix}
	    2\sqrt{s(x)/p(x)} & 0  \\
	    0 & 0
	 \end{pmatrix}h(x)+S(x)h(x).
      \end{align*}
      From \eqref{eq:pointwise} we obtain
      \begin{equation}
	 \label{h}
	 \|h-h_{1}\|_{\infty}
	 =:\left\| \begin{pmatrix}
	    G_{1}  \\ G_{2}
	 \end{pmatrix} \right\|_{\infty}
	 \leq \sum_{k=1}^{\infty}\frac{M^k}{k!}=e^{M}-1.
      \end{equation}
      Finally, by hypothesis, for any $\epsilon>0$, there exists $C\in \R$ such that
      $M(C) <\ln (1+\epsilon)$
      and therefore, by \eqref{eq:pointwise},
      \[|h(x)-h_{1}(x)|\leq e^{M(x)}-1<\epsilon\]
      for all $x\geq C.$
      It follows that
      \[G_{i}(x)\to0,\quad\mbox{when}\quad x\to b,\quad \mbox{for }i=1,2,\]
      hence by \eqref{h}
      \[h(x)=\begin{pmatrix}
	 G_1(x)  \\ G_2(x)+1
      \end{pmatrix}
      \to
      \begin{pmatrix} 0 \\ 1
      \end{pmatrix}
      \qquad
      \text{for } x\to b.
      \]
      Since $h$ is a solution for $\eqref{2.2}$, a solution $Z_{0}$ for \eqref{z-system2} is given by
      \begin{equation}
	 \label{eq:Zsolution}
	 Z_{0}(x)=\begin{pmatrix}
	    z_{0}(x) \\
	    p(x)u^{2}z_{0}'(x)
	 \end{pmatrix}
	 = e^{-\int_{a}^{x}\sqrt{s(t)/p(t)}\, \rd t}
	 \begin{pmatrix*}[r]
	    1 & 1 \\ 1 & -1
	 \end{pmatrix*}
	 \begin{pmatrix}
	    G_1(x)  \\
	    G_2(x)+1
	 \end{pmatrix}.
      \end{equation}
      Recall that $z=uy$.
      Setting $F_{1}:=G_{1}+G_{2}$ and $F_{2}:=G_{1}-G_{2}$ we obtain
      \[y(x)=(p(x)s(x))^{-1/4}e^{-\int_{a}^{x}\sqrt{s(t)/p(t)}\, \rd t}(1+F_{1}(x))\]
      and
      \[((ps)^{1/4}y)'(x)=(s/p)^{1/2}(x)e^{-\int_{a}^{x}\sqrt{s(t)/p(t)}\, \rd t}(F_{2}(x)-1)\]
      where we used that $(ps)^{1/4}y = z_0$ and $z_0'$ as given in~\eqref{eq:Zsolution}.
      \medskip

      In order to obtain a second solution of \eqref{eq:Wequation}, we set
      $\widehat V(x) := W(x) e^{-\int_{a}^{x}\sqrt{s(t)/p(t)}\, \rd t}$.
      It satisfies
      \begin{equation}
	 \label{2.2n}
	 \widehat V'(x)= \begin{pmatrix}
	    0 &0  \\
	    0& -2 \sqrt{ s(x)/p(x)}
	 \end{pmatrix}
	 \widehat V(x)+\frac{1}{2}(-u(x)(pu')'(x))
	 \begin{pmatrix*}[r]
	    1 & 1 \\ -1 & -1
	 \end{pmatrix*}\widehat V(x).
      \end{equation}
      As before, we note that the constant
      $\widehat{h}_{1}:= \begin{pmatrix} 1 \\ 0
      \end{pmatrix}$
      is a solution of the homogeneous equation
      \begin{equation}
	 \label{system2}
	 \widehat V'(x)=\begin{pmatrix}
	    0 & 0  \\
	    0 & -2\sqrt{s(x)/p(x)}
	 \end{pmatrix}\widehat V(x)
      \end{equation}
      and that
      \[\widehat{\Phi}(x)
      := \begin{pmatrix}
	 1 & 0  \\
	 0 & e^{-2\int_{a}^{x}\sqrt{s(t)/p(t)}\,\rd t}
      \end{pmatrix}\]
      is a fundamental system for \eqref{system2}.
      Again, for $a \le\widehat{x} < x$
      \[\|\widehat{\Phi}(x)\widehat{\Phi}^{-1}(\widehat{x})\|_{2}
      = \left\| \begin{pmatrix}
	 1 &0  \\
	 0& e^{-2\int_{\widehat{x}}^{x}\sqrt{s(t)/p(t))}\,\rd t}
      \end{pmatrix}\right\|_{2}=1\]
      and arguing as before, the operator $\widehat{F}$ defined by
      \[(\widehat{F}f)(x):= \widehat{h}_{1}+\widehat{\Phi}(x)\int_{a}^{x}\widehat{\Phi}^{-1}(t)S(t)f(t)\,\rd t\]
      yields a solution $\widehat h := \lim_{n\to\infty} F^{n}(\widehat h_1)$ for~\eqref{2.2n} which in turn leads to a solution $\widehat y$ of \eqref{2.1} which satisfies~\eqref{widey}.

In order to prove that $y$ and $\widehat{y}$ are linearly independent, we calculate the following Wronskian
\begin{align*}
    W((sp)^{1/4}y,(sp)^{1/4}\widehat{y})(x)&=(sp)^{1/4}\left(((sp)^{1/4}y)'(x)\widehat{y}(x)-y(x)((sp)^{1/4}\widehat{y})'(x)\right)\\
    &=\left(\frac{s(x)}{p(x)}\right)^{1/2}\left((F_{2}(x)-1)(1+\widehat{F}_{1}(x))-(1+\widehat{F}_{2}(x))(1+F_{1}(x))\right).
\end{align*}
As $F_{j}(x),\widehat{F}_{j}(x)\to0$ for $x\to b$ and $s(x)\neq 0$ for $x\in[a,b)$ the Wronskian is non-zero for all $x.$
\end{proof}

\section{Summability properties of solutions} \label{usbek}

Now we discuss  properties of the solutions of $\eqref{2.1}$ for the case $p=1$ and $b=\infty$.
In this section we will always assume that the conditions of the Theorem~\ref{t2.2} hold. Let $y$ and $\widehat{y}$ be the fundamental system of \eqref{2.1} as in \eqref{y} and \eqref{widey}.

\begin{remark}
   \label{rem:comparison}
   Note that $\re [ s(x)^{1/2} ] > 0$ for all $x\in [a,\infty)$.
   Since both $F_1$ and $\widehat F_1$ tend to zero for $x\to\infty$, it follows that
   $|y(x)|\leq |\widehat{y}(x)|$ for large $x$.
   In particular, $\widehat y\in L^2[a,\infty)$ implies $y\in L^2[a,\infty)$.
\end{remark}
We start with a helpful comparison.
\begin{lemma}
   \label{thm:comparison}
   Assume $p=1$ and $b=\infty$ and that the conditions of the Theorem~\ref{t2.2} are fulfilled. Let $y$ and $\widehat{y}$ be
   the fundamental system in \eqref{y} and \eqref{widey}.
   Let $\psi\ge 0$ be a measurable function on $[a, \infty)$.
   \begin{enumerate}[label=\upshape(\alph*)]
      \item\label{thm:comparison:a}
      If $\psi \in L^1(a_0,\infty)$ for some $a_0\ge a$ and 
			$$
      \limsup_{x\to \infty} \frac{1}{\psi(x)}
      \frac{e^{2\int_{a}^{x} \re [ s(t)^{1/2} ] \, \rd t}}{|s(x)|^{1/2}}<\infty,
			$$
      then $\widehat y\in L^2(a,\infty)$.

      \item\label{thm:comparison:b}
      If $\psi \notin L^1(a_0,\infty)$ for any $a_0\ge a$ and 
      $$
			\liminf_{x\to \infty} \frac{1}{\psi(x)}
      \frac{e^{2\int_{a}^{x}\re [ s(t)^{1/2} ] \, \rd t}}{|s(x)|^{1/2}} > 0,
			$$
      then $\widehat y\notin L^2(a,\infty)$.

      \item \label{criterio-s:a}
      All solutions of \eqref{2.1} are in $L^{2}[a,\infty)$ if and only if the following function is in $L^{2}[a,\infty)$:
      $$ x\mapsto s(x)^{-1/4}e^{\int_{a}^{x} (s(t))^{1/2}\,\rd t}.
      $$

   \end{enumerate}
\end{lemma}
\begin{proof}
   Note that for any nonnegative function $\psi$ we have that
   \begin{align*}
      | \widehat y(x) |^2
      = |s(x)|^{-1/2} |1+\widehat{F}_{1}(x)|^2
      e^{ 2\int_{a}^{x}\re [ s(t)^{1/2} ]\, \rd t}
      =
      \psi(x)
      \frac{ e^{ 2\int_{a}^{x}\re [ s(t)^{1/2} ]\, \rd t} }{ \psi(x) |s(x)|^{1/2}}
      |1+\widehat{F}_{1}(x)|^2 .
   \end{align*}
   Since $\lim_{x\to\infty} \widehat F_1(x) = 0$, we have for large enough $x$
   \begin{align*}
      \frac{1}{2}  \psi(x)
      \liminf_{\eta\to\infty}
      \frac{ e^{ 2\int_{a}^{\eta}\re [ s(t)^{1/2} ]\, \rd t} }{ \psi(\eta) |s(\eta)|^{1/2}}
      \le | \widehat y(x) |^2
      \le 2 \psi(x)
      \limsup_{\eta\to\infty}
      \frac{ e^{ 2\int_{a}^{\eta}\re [ s(t)^{1/2} ]\, \rd t} }{ \psi(\eta) |s(\eta)|^{1/2}}
   \end{align*}
   and \ref{thm:comparison:a} and \ref{thm:comparison:b} are proved. Assertion \ref{criterio-s:a} follows from the equality in \eqref{widey} and from Remark~\ref{rem:comparison}.
\end{proof}

The next corollary follows from Lemma~\ref{thm:comparison}~\ref{thm:comparison:a}, Remark~\ref{rem:comparison} and setting $\psi(x) := x^{-\rho}$.
\begin{corollary}\label{conmigo}
      If for some $\rho>1$
			$$
			\displaystyle \limsup _{x\to \infty}
      \frac{ x^{\rho} e^{2\int_{a}^{x}\re[s(t)^{1/2}]\, \rd t}}{ | s(x)|^{1/2}}<\infty,
			$$
      then all solutions of $\eqref{2.1}$ are in $L^{2}[a,\infty)$.
\end{corollary}

\begin{theorem}
   \label{criterio-s}
   Assume $p = 1$ and $b=\infty$ and that the conditions of the Theorem~\ref{t2.2} are fulfilled. Let $y$ and $\widehat{y}$ be
   the fundamental system as in \eqref{y} and \eqref{widey}.
   Then
	$$
	\widehat{y}\notin L^{2}[a,\infty)
	$$
	if one (or more) of the following conditions is satisfied.
   \begin{enumerate}[label=\upshape(\alph*)]
      \item \label{criterio-s:neu1}
      $|s|^{-1/2}\notin L^1[a,\infty)$.

      \item \label{criterio-s:b}
      The function $s$ is bounded, i.e.,  $s\in L^{\infty}$.

      \item \label{criterio-s:f}
      We have
			$$
			\int_a^\infty \re [s(x)^{1/2}]\, \rd t < \infty
      \quad \mbox{and} \quad
      \liminf_{x\to\infty} \frac{\re [s(x)^{1/2} ] }{ |s(x)|^{1/2} } > 0.
			$$

      \item \label{criterio-s:e}
      We have 
      $$\int_a^\infty \re [s(x)^{1/2}]\, \rd t = \infty
      \quad \mbox{and} \quad 
      \liminf_{x\to\infty} \frac{\left(\re [s(x)^{1/2}]\right)^{N}}{ | s(x)^{1/2} | } > 0
			\quad \mbox{for some } N\in\N.
			$$
			
			\item \label{criterio-s:c}
       $|\arg s|\leq \pi-\epsilon_{0}$ for some 
      $\epsilon_{0}>0$.

   \end{enumerate}
\end{theorem}
\begin{proof}

   \begin{enumerate}[label=\upshape(\alph*)]

   \item 
      By assumption, $\arg(s(t))\neq \pi$, hence
      $\re[ s(t)^{1/2} ] > 0$ for all $t\in [a,\infty)$.
      Moreover $\widehat F_{1}$ is a bounded function with $\lim_{x\to\infty}\widehat F_{1}(x) = 0$.
      Therefore, for any $c\in (0,1)$ we can take $a_0 > a$ such that 
	 $|1 + \widehat F_1(x)| > c$ for $x\in (a_0, \infty)$ and hence
      \begin{equation}
	 \label{eq:crit:neu1}
	 |\widehat{y}(x)|
	 = |s(x)|^{-1/4}|1 + \widehat F_1(x)|
	 e^{\int_{a}^{x}\re [s(t)^{1/2}]\, \rd t}
	 \ge  c|s(x)|^{-1/4},
      \qquad x > a_0.
      \end{equation}
      %
      Therefore the right hand side in \eqref{eq:crit:neu1} is not square integrable on $[a_0,\infty)$ which shows that $\widehat y \notin L^2[a,\infty)$.

      \item 
      If $s$ is bounded, then clearly $|s|^{-1/2}$ is not square integrable on $[a, \infty)$ and the claim follows from \ref{criterio-s:neu1}.

      \item 
      By assumption,
      $$
      \liminf_{x\to\infty} \frac{\re [s(x)^{1/2} ] }{ |s(x)|^{1/2} } := c > 0
      \quad \mbox{and} \quad
      0 < \int_a^\infty \re [ s(t)^{1/2} ] \, \rd t =: R < \infty.
      $$
      Hence the Lebesgue measure of the set
      $A := \{ t\in [a,\infty) : \re [ s(t)^{1/2} ] \le 1 \}$ is infinite.
      Note that the function $x\mapsto \int_{a}^{x}\re\sqrt{s(t)}\, \rd t$
      is non-decreasing and for large enough $x$, we have that
	$$
		\int_{a}^{x}\re [s(t)]^{1/2}\, \rd t \geq \frac{R}{2}.
	$$
      We conclude
      \begin{equation*}
  |\widehat y(x)|^2 =
  |1+F_1(x)|^2 \frac{e^{2\int_{a}^{x}\re [s(t)]^{1/2}\, \rd t}}{|s(x)|^{1/2}}
  \ge
  \frac{ce^{R} }{2}  \frac{1}{ \re [ s(x)^{1/2} ] }
  \ge
  \frac{ce^{R} }{2}  \chi_{A}(x)
      \end{equation*}
      where $\chi_A$ is the characteristic function of $A$.
      Since $\frac{ce^R}{2}>0$ it follows that $\widehat y\notin L^2[a,\infty)$.

      \item 
      We define
      \begin{equation*}
      \psi(x) := \frac{ e^{2\int_{a}^{x}\re [ s(t)^{1/2} ]\, \rd t} }{ (\re s(x)^{1/2})^{N} },\qquad x\in [a,\infty).
       \end{equation*}
      Then
      \begin{equation*}
			\psi(x)^{-1}\frac{e^{2\int_{a}^{x}\re [ s(t)^{1/2} ] \, \rd t}}{|s(x)|^{1/2}}=
	 \frac{ (\re s(x)^{1/2})^{N} }{|s(x)|^{1/2}}
      \end{equation*}
      and by Lemma~\ref{thm:comparison}~\ref{thm:comparison:b} it suffices to show that
      $\psi \notin L^1[a,\infty)$. If we set 
			$$
			g(x) := \int_a^x \re [s(t)^{\frac{1}{2}}]\,\rd t,
			$$
      then $g$ satisfies the differential equation
      \begin{align*}
	 \frac{e^{2g (x)}}{( g '(x))^{N}}&= \psi(x).
      \end{align*}
      For
			$$
			G(x):= \int_{a}^{x}(\psi(t))^{-1/N}\,\rd t
			$$
			we obtain
      \begin{align*}
	 e^{-\frac{2}{N} g(x)}&= 1- \frac{2}{N} G(x).
      \end{align*}
      By assumption, $g(x)\to \infty$ for $x\to \infty$, so 
			$G(x)\to N/2$ and $\int_{a}^{\infty}\psi^{-1/N}(t)\,\rd t<\infty$.
   Similarly as above, the Lebesgue measure of the set
      $\{ t\in [a,\infty) : \psi^{-1/N}(t) \le 1 \}$ is infinite.
			Hence 
			$$
			\int_{a}^{\infty}\psi^{1/N}(t)\,\rd t=\infty,
			$$
			which implies that $\int_{a}^{\infty}\psi(t)\,\rd t$ diverges.
			
      \item 
      The assumption on $s$ implies that
      $|\arg (s(x))^{1/2} | < (\pi - \epsilon_0)/2$.
      Setting $c := \cos\left(\frac{\pi-\epsilon_{0}}{2}\right)$,
      we have that $\re s(x)^{1/2} = |s(x)|^{1/2}  \cos \arg (s(x))^{1/2}$ and
      \[ \frac{ \re s(x)^{1/2} }{ |s(x)|^{1/2} }  \ge c , \qquad x\in [a,\infty). \]
      If $\int_a^\infty \re [ s(t)^{1/2} ] \, \rd t = \infty$, then the claim follows from \ref{criterio-s:e} with $N=1$.
      If on the other hand $\int_a^\infty \re [ s(t)^{1/2} ] \, \rd t =: R < \infty$,
      the claim follows from \ref{criterio-s:f}.
      \qedhere

   \end{enumerate}
\end{proof}

\begin{remark} If, in addition to the assumptions in
Theorem \ref{criterio-s} \ref{criterio-s:b}, we have $\re [ s^{1/2} ] \in L^{1}[a,\infty)$, then  there exists $K>0$ such that
      \begin{equation*}
  |y(x)|
  \ge K|s(x)|^{-1/4}|1+F_1(x)|
  \ge K \|s\|_\infty^{-1/4}|1+F_1(x)|,
      \end{equation*}
      which implies that $y\notin L^{2}[a,\infty)$.
\end{remark}

\section{Sturm-Liouville equation on a ray} \label{Russisch}

In this section the results of the previous section are used to investigate a second order differential equation defined on a ray in $\mathbb C$. This is motivated
by the recent intensive research connected with $\mathcal P \mathcal T$-symmetric
Hamiltonians, cf.\ \cite{bender}.
Let $\Gamma$ be a ray with angle $\phi \in (-\pi/2,\pi/2)$,
\begin{equation}
\label{Ozuna}
\Gamma:=\{z\in\C:z=xe^{i\phi}, x\in [a,\infty)\}
\end{equation}
for some $a\geq 0$.
Our main interest is to obtain a Weyl criterion for the differential equation
\begin{equation*}
-y(z)''+ \textbf{q}(z)y(z)=\lambda y(z),
\qquad z\in \Gamma
\end{equation*}
where 
 $\textbf{q}:\Gamma \to \C$ is locally integrable.
Setting
\begin{equation}\label{OzunaII}
v(x):=y(z(x)) \mbox{ and } q(x):=\textbf{q}(z(x)) 
\mbox{ with } z(x):=xe^{i\phi},\quad x\in [a,\infty),
\end{equation}
 we obtain the Sturm-Liouville problem
\begin{equation}
    \label{pt ecuation}
    -e^{-2i\phi}v''(x)+q(x)v(x)=\lambda v(x),
    \qquad x\in [a,\infty).
\end{equation}
In order to describe the solutions of \eqref{pt ecuation}, we re-write it in the form
\begin{equation}
    \label{pt-2}
    v''(x)=e^{2i\phi}\left(q(x)-\lambda \right)v(x),
    \qquad x\in [a,\infty).
\end{equation}
If we set
\begin{equation}\label{Kiew}
   s_\lambda(x) := e^{2i\phi}\left(q(x)-\lambda \right),
   \qquad x\in [a,\infty),
\end{equation}
then \eqref{pt-2} is in the form of \eqref{2.1} with $p = 1$ and solutions for \eqref{pt ecuation} are obtained from Theorem~\ref{t2.2}.
\begin{theorem}
   \label{t3.1}
   The differential equation \eqref{pt ecuation} with $\phi\in (-\frac{\pi}{2},\frac{\pi}{2})$ and $\lambda\in \C$ such that
   \begin{itemize}
   \addtolength{\itemsep}{.5\baselineskip}
    \item[(i)]$q, q'\in \textnormal{AC}_{\textnormal{loc}}[a,\infty)$,
      \item[(ii)] $\lambda\notin Q = \cco\{e^{-2i\phi}r+q(x) \, :\, 0<r<\infty,\, x\in [a,\infty)\}$ and
      \item[(iii)] 
      $\displaystyle
      M:=\int_a^\infty\left|\frac{5(e^{2i\phi}q'(x))^{2}}{16(e^{2i\phi}(q(x)-\lambda))^{5/2}}-\frac{e^{2i\phi}q''(x)}{4(e^{2i\phi}(q(x)-\lambda))^{3/2}}\right|dx<\infty
      $
   \end{itemize}
   has a fundamental system  $\{y,\widehat{y}\}$ of the form
   \begin{alignat}{3}
      \label{solutions}
      y(x) &= s_\lambda(x)^{-1/4}e^{-\int_{a}^{x} (s_\lambda(t))^{1/2} \,\rd t}(1+F_{1}(x)),
      && \qquad x\in[a,\infty),\\
      \label{Gesundbrunnen}
      \widehat{y}(x) &= s_\lambda(x)^{-1/4}e^{\int_{a}^{x}(s_\lambda(t))^{1/2} \,\rd t}(1+\widehat{F}_{1}(x)),
      && \qquad x\in[a,\infty),
   \end{alignat}
   with $|\widehat{F}_{1}(x)|,|F_{1}(x)|\to0$ when $x\to +\infty $, $\|F_{1}\|_{\infty}\leq 2e^{M}-2$ and $\|\widehat{F}_{1}\|_{\infty}\leq 2e^{M}-2.$
\end{theorem}
\begin{proof}
We show that the assumptions of Theorem~\ref{t2.2} for the function $s_\lambda$ in \eqref{Kiew}
are fulfilled.

We have $p = 1$
and it follows from (ii) (by sending $r$ to zero) that $q(x) \neq \lambda$ for all
$x\in [a,\infty]$, hence $s_\lambda(x) \neq 0$. 
Assume that there exists $x_0\in [a,\infty]$ with
$s_\lambda(x_0)\in (-\infty,0]$. 
Therefore, by definition of $s_\lambda$ in~\eqref{Kiew},
$$
\lambda = q(x_0) -s_\lambda(x_0)e^{-2i\phi} \in Q,
$$
a contradiction to (ii). Hence,
\begin{equation}\label{Bischleben}
\arg s_\lambda(x) \neq \pi \quad \mbox{for all}\quad
x\in[a,\infty). 
\end{equation}
Moreover, by (i), $s_\lambda^{-1/4}$ and  $(s_\lambda^{-1/4})^\prime$ are in 
$\textnormal{AC}_{\textnormal{loc}}[a,\infty)$ and we have
$$
s_\lambda^{-1/4}(s_\lambda^{-1/4})^{\prime\prime}= s_\lambda^{-1/4}
\left(
\frac{5}{16}(s_\lambda^\prime)^2s_\lambda^{-9/4} - \frac{1}{4} s_\lambda^{\prime\prime}
s_\lambda^{-5/4}
\right).
$$
Now, (iii) implies that 
$$
s_\lambda^{-1/4}(s_\lambda^{-1/4})^{\prime\prime} \in L^{1}(a,\infty)
$$
and we obtain the desired solutions from Theorem~\ref{t2.2}.
\end{proof}
We prove a limit point/limit circle criteria for 
the differential expression \eqref{pt ecuation}. This is our main result. Note that, if
$\cco\{e^{-2i\phi}r+q(x):0<r<\infty,\, x\in [a,\infty)\}\neq \C$,
then the equation \eqref{pt ecuation} is in the Weyl trichotomy described in the Definition~\ref{d1.1}. 

\begin{theorem}\label{t3.2}
   Assume that conditions (i)--(iii) of Theorem~\ref{t3.1} are satisfied for some $\phi\in (-\frac{\pi}{2},\frac{\pi}{2})$ and $\lambda\in \C\setminus Q$.
   Then the following is true.
   \begin{enumerate}[label=\upshape(\alph*)] 
      
      \item
      The equation \eqref{pt ecuation} is in the limit point~I case if one of the following conditions is fulfilled:

      \begin{enumerate}[label=\upshape(\alph*)] 
      
         \item[\upshape ($\alpha$)] 
         We have $|q-\lambda|^{-1/2}\notin L^1[a,\infty)$ or
      
	  \item[\upshape ($\beta$)] 
        $q \in L^{\infty}[a,\infty)$ or
      
	\item[\upshape ($\gamma$)] 
        we have
        $$
        \int_{a}^{\infty}\re\left[ e^{i\phi} (q(x)-\lambda)^{1/2} \right]dx<\infty\quad \text{and} \quad   \displaystyle\liminf_{x\to\infty} \frac{\re\left[ e^{i\phi} (q(x)-\lambda)^{1/2} \right]}{\left|
        q(x)-\lambda\right|^{1/2}}>0 \quad \text{or}
        $$
      
        \item[\upshape($\delta$)]
        for some $N\in\mathbb N$ we have
        $$
        \int_{a}^{\infty}\re\left[e^{i\phi} (q(x)-\lambda)^{1/2} \right]dx
        =\infty \quad \text{and} \quad \liminf_{x\to\infty} \frac{\left(\re \left[ e^{i\phi}
        (q(x)-\lambda)^{1/2}\right] \right)^{N}}{\left|q(x)-\lambda\right|^{1/2}}>0 \quad \text{or}
        $$
        
        \item[\upshape($\epsilon$)]
        we have  for some $\epsilon_{0}>0$
        $$
        |\arg [e^{i\phi} (q(x)-\lambda)^{1/2} ] |\leq \pi-\epsilon_{0}.
        $$ 
      \end{enumerate}

      In the cases $(\alpha)$--$(\epsilon)$ this classification is independent of the choice of $\lambda$, cf.\ Remark~\ref{CanNeverLoveAgain}.
      \medskip

      \item
      The equation \eqref{pt ecuation} is in the limit circle case if for some $\rho>1$
      $$
      q\in  L^1_{\mathrm u}(a,b) \quad \mbox{and} \quad
    \limsup\limits_{x\to \infty}\frac{x^{\rho}e^{2\int_{a}^{x}\re
   \left[ e^{i\phi} (q(x)-\lambda)^{1/2} \right]\,\rd t}}{|
   q(x)-\lambda|^{1/2}}<\infty.
      $$
      This classification is independent of the choice of $\lambda$, 
      cf. Remark~\ref{CanNeverLoveAgain}.
   \end{enumerate}
\end{theorem}
\begin{proof}
   By Theorem~\ref{t3.1} we know that there exist two linearly independent solutions $y$ and $\widehat y$. Cases $(\alpha)$--$(\epsilon)$ follow directly from Theorem~\ref{criterio-s}.
   It remains to show item (b). By Corollary \ref{conmigo} all solutions of $\eqref{2.1}$ are in $L^{2}[a,\infty)$.
   It is easy to see that the statements in the appendix of \cite{BST19} also hold true for non-real (i.e.\ complex-valued) potentials. 
   Then multiplying \eqref{pt ecuation} by $e^{2i\phi}$ and applying \cite[Lemma~A.2~(i)]{BST19} one obtains that all solutions fulfill \eqref{1.3}. 
\end{proof}

Theorem \ref{criterio-s} \ref{criterio-s:f} and \ref{criterio-s:e} show that a necessary condition for \eqref{pt ecuation} to be in the limit circle case at $\infty$ is that 
$\liminf_{x\to\infty} \frac{\re [s_\lambda(x)^{1/2} ] }{ |s_\lambda(x)|^{1/2} } = 0$ where $s_\lambda(x) = e^{2i\phi} (q(x) - \lambda)$ as in \eqref{Kiew}, but in general it is not sufficient, see Theorem~\ref{thm:criteriaReIm} and Example~\ref{ex:polynomials}.
The next theorem gives sufficient criteria for limit point~I, limit point~II and limit circle under the assumption that the limit inferior is actually a limit.
\smallskip

For functions $f, g$ defined on $(0,\infty)$ we will use the notation $f \sim g$ if there exist constants $m, M> 0$ such that $m g(x) \le f(x) \le M g(x)$ for all large enough $x$.
In particular, if $f\sim g$, then $f$ is integrable at infinity if and only if so is $g$.
\smallskip

In the next lemma, we relate the asymptotics of $s_\lambda$ and $q$.
\begin{lemma}
   Assume that conditions (i)--(iii) of Theorem~\ref{t3.1} are satisfied for some $\phi\in (-\frac{\pi}{2},\frac{\pi}{2})$ and $\lambda\in \C$.
   Let $\lambda\in\C\setminus Q$ where $Q = \cco\{ q(x) + rp : 0 < r < \infty,\, x\in [a,\infty) \} $ is defined in \eqref{1.2}.
   Recall that here $p=e^{-2i\phi}$.
   As in \eqref{Kiew}, we set $s_\lambda := p^{-1}(q - \lambda)$.
   If $\lim_{x\to\infty} |q(x)| = \infty$, or equivalently 
   $\lim_{x\to\infty} |s_\lambda| = \infty$, 
   then the limit
   \begin{equation}
      \label{eq:defq0}
      q_0 
      := 
      \lim_{x\to\infty} \frac{\re q(x)^{1/2}  }{ |q(x)|^{1/2} }
   \end{equation}
   exists if and only if
    \begin{equation}
      \label{eq:defs0}
      \snull 
      := 
      \lim_{x\to\infty} \frac{\re s_\lambda(x)^{1/2}  }{ |s_\lambda(x)|^{1/2} }
   \end{equation} 
   exists for one, and hence for all $\lambda\in\C\setminus Q$ and it is independent of $\lambda$.
   In this case, we have that
   \begin{equation}
       q_0 = \snull \cos\phi + \sqrt{1-\snull^2}\, \sin\phi,
       \qquad
       \snull = q_0 \cos\phi - \sqrt{1-q_0^2}\, \sin\phi.
   \end{equation}
\end{lemma}
\begin{proof}
    Clearly, 
    $\lim_{x\to\infty} |q(x)| = \infty$ if and only if
    $\lim_{x\to\infty} |s_\lambda| = \infty$ 
    for one, or equivalently, for all $\lambda\in\C$.
    Hence, assume that the limit in \eqref{eq:defs0} exists for one $\lambda\in \mathbb C$. Therefore,
 \begin{align*}
\snull := \lim_{x\to\infty} \frac{\re s_\lambda(x)^{1/2}  }{ |s_\lambda(x)|^{1/2} } =  \lim_{x\to \infty}
      \cos \arg [ s_\lambda(x)]^{1/2} 
    \end{align*} 
The expression $ [ s_\lambda(x)]^{1/2} $ 
is in the open right half-plane, see \eqref{Bischleben},
where the arg-function is continuous. Therefore, we have
the existence of 
$$
 \lim_{x\to \infty} \arg [ s_\lambda(x)]^{1/2}.
$$
Note that 
\begin{align*}
  \frac{\re q(x)^{1/2}  }{ |q(x)|^{1/2} }
  & =
  \frac{\re [ p s_\lambda(x) + \lambda ]^{1/2} }{ | ps_\lambda(x)+\lambda|^{1/2} }
  = \cos \arg [p s_\lambda(x) + \lambda]^{1/2}
  = \cos \arg \left[ p \frac{s_\lambda(x)}{|s_\lambda(x)|} + \frac{\lambda}{|s_\lambda(x)|} \right]^{1/2} 
  \\
   & = \cos \left(-\phi+\arg \left[  \frac{s_\lambda(x)}{|s_\lambda(x)|} + \frac{p^{-1}\lambda}{|s_\lambda(x)|} \right]^{1/2}
  \right) 
  \\
  & =
   \cos(-\phi) \cos \arg \left[  \frac{s_\lambda(x)}{|s_\lambda(x)|} + \frac{p^{-1}\lambda}{|s_\lambda(x)|} \right]^{1/2} -
   \sin(-\phi) \sin \arg \left[  \frac{s_\lambda(x)}{|s_\lambda(x)|} + \frac{p^{-1}\lambda}{|s_\lambda(x)|} \right]^{1/2}.
\end{align*} 

Since $|s_\lambda(x)|\to\infty$ for $x\to\infty$, this shows that
\eqref{eq:defq0} exists and that 
    \begin{align*}
    q_0=  \lim_{x\to \infty}
      \frac{\re q(x)^{1/2}  }{ |q(x)|^{1/2} }
      & = \lim_{x\to \infty}\left(
      \cos(-\phi)\cos\arg [s_\lambda(x)]^{1/2}  
      -
      \sin(-\phi)\sin\arg [s_\lambda(x)]^{1/2}  \right)
      \\
      & = 
      \snull\cos(\phi) + \sqrt{1-\snull^2}\, \sin(\phi).
   \end{align*} 
The above reasoning is valid for any $\lambda \in \mathbb C\setminus Q$. As $q_0$ does not depend on $\lambda$, the same is true for $\snull$.
A similar calculation shows that the existence of the limit in \eqref{eq:defq0} implies the existence of the limit in \eqref{eq:defs0} for any $\lambda\in \C\setminus Q$ and that the latter limit does not depend on $\lambda$.
\end{proof}

\begin{theorem}
   \label{thm:criteriaReIm}
   Assume that the conditions (i)--(iii) from Theorem~\ref{t3.1} are satisfied for some $\phi\in (-\frac{\pi}{2},\frac{\pi}{2})$ and $\lambda\in\C\setminus Q$ where $Q = \cco\{ q(x) + rp : 0 < r < \infty,\, x\in [a,\infty) \} $ is defined in \eqref{1.2}.
   Recall that here $p=e^{-2i\phi}$.
   As in \eqref{Kiew}, we set $s_\lambda := p^{-1}(q - \lambda)$.
   In addition, we assume that $\lim_{x\to\infty} |q(x)| = \infty$ and we assume that $q_0$ and $\snull$ as defined in \eqref{eq:defq0} and \eqref{eq:defs0} exist.
   If 
   $\snull = 0$ (or equivalently, $q_0 = \sin\phi)$, then $\im[(s_\lambda(x))^{1/2}]\to \pm\infty$ 
   and for any $\lambda\notin Q$, the only admissible angle $\theta$ such that \eqref{condition} holds, are $\theta = \pm \frac{\pi}{2} + 2\phi $.
   The Weyl classification of equation \eqref{pt ecuation} at $\infty$ is then

   \begin{enumerate}[label=\upshape(\roman*)] 
      \item \label{criterion:LPIa}
      limit point case~I if $\snull \neq 0$.

      \item \label{criterion:LPIb}
      limit point case~I if $\snull= 0$ and 
      $\displaystyle \frac{ e^{2\int_a^x \re [s_\lambda(t)^{1/2}] \, \rd t} }{ \im[ s_\lambda(x)^{1/2} ] }\notin L^1[a,\infty)$.

      \item \label{criterion:LPII}
      limit point case~II if $\snull = 0$, $\re[ s_\lambda^{1/2}]\notin L^1[a,\infty)$ and 
      $\displaystyle \frac{ e^{2\int_a^x \re [s_\lambda(t)^{1/2}] \, \rd t} }{ \im[ s_\lambda(x)^{1/2} ] }\in L^1[a,\infty)$.

      \item \label{criterion:LC}
      limit circle case if $\snull = 0$, $\re[ s_\lambda^{1/2}] \in L^1[a,\infty)$ and $(\im [ s_\lambda^{1/2} ])^{-1} \in L^1[a,\infty)$.

   \end{enumerate}
   The classification is independent of $\lambda\notin Q$.

\end{theorem}
\begin{proof}
   Recall that 
   if $\snull$ exists in $\R$
   then the equality in \eqref{eq:defs0} holds independently of the choice of $\lambda$.
   Let 
   $$
   f := \re[ s_\lambda^{1/2}]\quad \mbox{and} \quad g := \im [s_\lambda^{1/2}].
   $$
   Note that $f > 0$ since $\arg(s)\neq \pi$, see \eqref{Bischleben}.
   The assertion \ref{criterion:LPIa} follows immediately from ($\gamma$) and ($\delta$) in Theorem~\ref{t3.2}.
   Now let us assume that $\snull = 0$.
   Then $|g(x)| \sim |s_\lambda(x)|$ and $\frac{|f(x)|}{|g(x)|} \to 0$.
   Since $\lim_{x\to\infty} |s_\lambda(x)| = \infty$, it follows that $g(x)\to \infty$ or
   $g(x)\to -\infty$ for $x\to \infty$.
   Without restriction we may assume that $g(x)\to\infty$.
   Hence $\im s_\lambda > 0$ for $x$ large enough and 
   $\arg [p^{-1} q(x)] \to \pi$ for $x\to\infty$.
   Therefore
   \begin{align*}
   p^{-1}Q 
   &= \cco\{ p^{-1} q + r : 0 < r < \infty,\, x\in [a,\infty) \}
   = \{ z\in\C : A_{\min} \le \re z \le A_{\max} \}
   \end{align*}
   is an infinite strip parallel to the real axis
   where $A_{\min}$ and $A_{\max}$ are the minimum, respectively maximum of the bounded set $\{ \re( p^{-1} q(x) : x\in [a,\infty)\}$.
   Consequently, $Q$ is an infinite strip which has the angle $-2\phi$ with the real axis, see also Figure~\ref{fig:Q}.
   
   Let $\lambda\in\C\setminus Q$. 
   If $p^{-1}\lambda$ is in the half plane below the strip $p^{-1}Q$, then 
   $\theta = 2\phi - \pi/2$ is the only possible angle for which $e^{ i\theta}Q$ is a subset of some right half plane and $e^{i\theta}\lambda$ is to its left.
   If on the other hand $p^{-1}\lambda$ is in the half plane above the strip $p^{-1}Q$, then 
   $\theta = 2\phi + \pi/2$ is the only possible angle for which $e^{ i\theta}Q$ is a subset of some right half plane and $e^{i\theta}\lambda$ is to its left.
   
   This already shows that the first integral in \eqref{1.3} is $0$ because $\re[e^{i\theta} p] = \re [e^{\pm i\pi/2+2i\phi} e^{-2i\phi}] = 0$.

   From the asymptotics \eqref{Gesundbrunnen} for $\widehat y$ we obtain
   \begin{equation}
      \label{eq:wideyAsymptotics}
      |\widehat y(x)|^2
      \sim \left| \frac{e^{2\int_a^x f(t) + i g(t)\,\rd t}}{ (s_\lambda(x))^{1/2} }  \right|
      \sim \frac{e^{2\int_a^x f(t)\,\rd t}}{ |s_\lambda(x)|^{1/2} }
      \sim \frac{e^{2\int_a^x f(t)\,\rd t}}{ g(x) }
      \qquad \text{ for } x\to\infty.
   \end{equation}
   Hence $\widehat y\notin L^2[a,\infty)$ under the assumptions in \ref{criterion:LPIb} and consequently equation \eqref{pt ecuation} is in the limit point~I case at $\infty$.
   On the other hand, $\widehat y\in L^2[a,\infty)$ under either of the assumptions in \ref{criterion:LPII} and \ref{criterion:LC} and therefore equation \eqref{pt ecuation} is in the limit point~II or limit circle case at $\infty$.
   Which case prevails, depends on whether \eqref{1.3} is satisfied, that is, whether 
   $\re( e^{i\theta} q) |\widehat y|^2$ is integrable.
   Observe that, by \eqref{eq:wideyAsymptotics},
   \begin{align*}
      | \re( e^{i\theta} (q(x)-\lambda) )|\, |\widehat y(x)|^2
      & = | \im e^{2i\phi} (q(x)-\lambda )|\,  |\widehat y(x)|^2
      = | \im( s_\lambda(x) ) |\,  |\widehat y(x)|^2
      \\
      & = 2 f(x) g(x)\, |\widehat y(x)|^2
      \sim 2 f(x) e^{2\int_a^x f(t)\,\rd t}\ .
   \end{align*}
   The latter is integrable if and only if $f$ is integrable.
   This concludes the proof of assertions \ref{criterion:LPII} and \ref{criterion:LC}.
\end{proof}

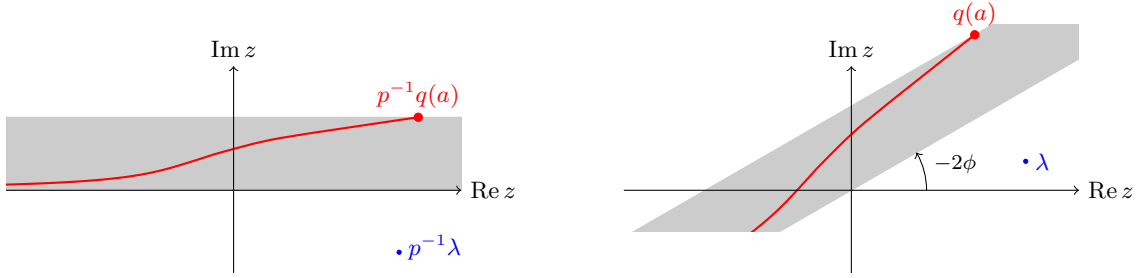
\begin{figure}
   \centering

   \tikzmath{
   \startpoint = 0.85;
   \mtwophi = 30;  
   }
   \tikzset{
   declare function={ 
   rootsRe(\x) = (\x)^(-5); 
   rootsIm(\x) = (\x)^(1.5); 
   sRe(\x) = (rootsRe(\x))^2 - (rootsIm(\x))^2;
   sIm(\x) = -2(rootsRe(\x)) * (rootsIm(\x));
   }
   }

   \begin{tikzpicture}[scale=.55, baseline=(O)] 

      \coordinate (O) at (0,0);
      \begin{scope}
	 \clip (-5.5,-1) rectangle (5.5, 4);

	 \draw [colQ, fill] (-5.5,0) rectangle ({sRe(\startpoint)+4}, {sIm(\startpoint)});

	 \draw [red, thick] plot[domain=\startpoint:2.9, smooth] ( {sRe(\x)}, {sIm(\x)});
      \end{scope}
      
     \coordinate (la) at (4,-1.5);
      \draw [blue, fill] (la) circle  [radius=.05] node [right, yshift=2] {\small$p^{-1}\lambda$};
  
      \draw[->] (-5.5,0)--(5.5,0) node [right] {\small$\re z$};
      \draw[->] (0,-2)--(0,3) node [above] {\small$\im z$};
      \draw [red, fill] ({sRe(\startpoint)}, {sIm(\startpoint)}) circle  [radius=.1] node [above] {\small$p^{-1}q(a)$};

   \end{tikzpicture}
   %
   \hspace{\fill}
   \begin{tikzpicture}[scale=.55, baseline=(O)] 
      \coordinate (O) at (0,0);

      \begin{scope}
	 \clip (-5.5,-1) rectangle (5.5, 4);

	 \draw [colQ, fill, rotate=\mtwophi] (-5.5,0) rectangle ({sRe(\startpoint)+4}, {sIm(\startpoint)});

	 \draw [red, thick, rotate=\mtwophi] plot[domain=\startpoint:2.9, smooth] ( {sRe(\x)}, {sIm(\x)});
      \end{scope}

     \begin{scope}[rotate=\mtwophi]
      \coordinate (la) at (4,-1.5);
      \draw [blue, fill] (la) circle  [radius=.05] node [right] {\small$\lambda$};
     \end{scope}
  
      \draw[->] (-5.5,0)--(5.5,0) node [right] {\small$\re z$};
      \draw[->] (0,-2)--(0,3) node [above] {\small$\im z$};
      \draw [red, fill, rotate=\mtwophi] ({sRe(\startpoint)}, {sIm(\startpoint)}) circle  [radius=.1] node [above] {\small$q(a)$};

      \coordinate (0) at (0,0);
      \coordinate (A) at (3,0);
      \coordinate (B) at (\mtwophi:2);

      \draw pic[->, draw, angle radius=1.0cm,"\footnotesize $-2\phi$" shift={(8mm,2mm)}] {angle=A--0--B};

   \end{tikzpicture}

   \caption{
   Illustration for the set $Q$ in Theorem~\ref{thm:criteriaReIm}.
   The red lines are the curves $p^{-1}q$ and $q$ respectively and the grey regions are $p^{-1} Q$ and $Q$.
   Note that $p^{-1}Q$ is a strip parallel to the real axis.
   }
   \label{fig:Q}
\end{figure}

We finish this section with various examples  which illustrate Theorem~\ref{thm:criteriaReIm}. Some other classes of examples are in
\cite[Section 5]{MC} and \cite[Chapter III Thm.\ 10.28]{EE}.

\begin{exam}
   \label{ex:polynomials}
   Consider $\Gamma$ as in \eqref{Ozuna} with $a>0$ and $\phi\in (-\pi/2, 0]$. 
   Moreover, we set $\lambda=0$ and $q(x):= e^{-2i\phi}(x^{-m} + i x^n)^2$ with 
   $m \ge 0$ and $n > 0$.
   
   \begin{enumerate}[label=\upshape(\roman*)]
   
      \item\label{ex:polynomials:LPI}
      The equation~\eqref{pt ecuation} is in the limit point I case if $0 \le m < 1$
      
      \item\label{ex:polynomials:LPII}
      The equation~\eqref{pt ecuation} is in the limit point II case if $m = 1$ and $n > 3$.

      \item\label{ex:polynomials:LC}
      The equation~\eqref{pt ecuation} is in the limit circle case if $m > 1$ and $n>1$.

   \end{enumerate}

   \begin{proof}
      Straightforward calculations show that the hypotheses from  Theorem~\ref{t3.1} and Theorem~\ref{thm:criteriaReIm} are satisfied.
      With the notation from 
      Theorem~\ref{thm:criteriaReIm} we have that 
      $$
      s_\lambda(x) = e^{2i\phi}q(x) = (x^{-m} + i x^n)^2
      \quad \mbox{hence} \quad s_\lambda(x)^{1/2} =  x^{-m} + i x^n  =: f(x) + i g(x)
      $$
    and 
      $$
      \snull = \lim_{x\to\infty} \frac{\re s_\lambda(x)^{1/2}  }{|s_\lambda(x)|^{1/2}} 
      = \lim_{x\to\infty} \frac{ x^{-m} }{|x|^{n}} = 0.
      $$
      We obtain by \eqref{Gesundbrunnen}
      \begin{equation*}
	 |\widehat y(x)|^2 \sim \frac{1}{|s_\lambda(x)|^{1/2}} e^{2\int_a^x f(t)\, \rd t}
	 \sim x^{-n} e^{2\int_a^x t^{-m}\, \rd t}
	 \sim
	 \begin{cases}
	    x^{-n} e^{2\frac{1}{1-m} x^{1-m}}, &\quad\text{if } m\neq 1,\\
	    x^{2-n} , &\quad\text{if } m= 1.
	 \end{cases}
      \end{equation*}

      Now the assertions 
      follow immediately from Theorem~\ref{thm:criteriaReIm}:
      \begin{enumerate}[label=\upshape(\roman*)]

	 \item
	 If $0 \le m < 1$ or if $m=1$ and $0<n<3$, then $\widehat y \notin L^2[a,\infty)$, hence the equation~\eqref{pt ecuation} is in the limit point~I case.

	 \item
	 If $m = 1$ and $n> 3$, then $\re[ (s_\lambda(x)^{1/2})] = x^{-1}$ does not belong to $L_1[a,\infty)$ while
	 $$
  \frac{\exp(2\int_a^x \re[ s_\lambda(t)^{1/2}]\,\rd t ) }{\im[ (s_\lambda(x))^{1/2} ]}  = \frac{x^{-2}}{x^{n}} 
  $$ 
  does. Therefore, according to Theorem~\ref{thm:criteriaReIm} (iii), the equation~\eqref{pt ecuation} is in the limit point II case.

	 \item
	 If $m > 1$ and $n>1$ then both $\re[ (s_\lambda(x)^{1/2})] = x^{-m}$ 
	 and $(\im[ (s_\lambda(x)^{1/2})])^{-1} = x^{-n}$ belong to $L_1[a,\infty)$,
	 hence the equation~\eqref{pt ecuation} is in the limit circle case according to Theorem~\ref{thm:criteriaReIm} (iv).
	 \qedhere
      \end{enumerate}
   \end{proof}
   
   If we consider the potential $q(x) = e^{-2i\phi}(ct^{-1} + it^n)^2$ with $c> 0$, then by \eqref{Gesundbrunnen}
   $|\widehat y(x)|^2 \sim  x^{2c-n}$ for $x\to\infty$.
   Hence the differential equation is in the limit point~I case if $n \le 2c+1$.
   It is in the limit point~II case if $n > 2c+1$.
   The case $m=1, n=3, c=3/4$ is (asymptotically) the example $q(x) = x^{6} - \frac{3}{2}i x^2,\ x\in (1,
   \infty)$ from Sims \cite[p. 257]{SIMS} for the limit point~II case.
\end{exam}

\begin{exam}
   \label{e4.3}
   Consider $\Gamma$ as in \eqref{Ozuna} with $\phi=0$. Moreover we set $\lambda=0$ and 
   $$
   q(x):= h(x)+iM \quad \mbox{for some}\quad M\in\R\setminus\{0\}
   $$
   and real valued $h$ with 
   $h, h'\in \textnormal{AC}_{\textnormal{loc}}[a,\infty)$ such that item (iii) of 
   Theorem~\ref{t3.1} is satisfied.
   \begin{enumerate}[label=\upshape(\roman*)]
   
      \item\label{thm:e4.3:bounded}
      If  $\|h\|_\infty < \infty$, then  \eqref{pt ecuation} is in the limit point~I case.

      \item\label{thm:e4.3:infinity}
      If $\lim\limits_{x\to\infty} h(x) = -\infty$, then \eqref{pt ecuation} is in the limit point~I case if  and only if $h^{-1/2}\notin L^{1}[a,\infty)$. 
      
      \item\label{thm:e4.3:infinityLC}
      If $\lim\limits_{x\to\infty} h(x) = -\infty$, then \eqref{pt ecuation} is in the limit circle case if  and only if $h^{-1/2}\in L^{1}[a,\infty)$. 
      
   \end{enumerate}

   \begin{proof}
      Note that
      \[\cco\{r+q(x)\, :\, x\in [a,\infty),\,  0<r<\infty\} \subseteq \{t+iM : t\in\R\},\]
      with equality in the cases \ref{thm:e4.3:infinity} and \ref{thm:e4.3:infinityLC}.
      The claim in \ref{thm:e4.3:bounded} follows directly from Theorem~\ref{criterio-s}\ref{criterio-s:b}.
      Now assume that $\lim\limits_{x\to\infty} h(x) = -\infty$.
      Note that for $x$ large enough so that $h(x)<0$ we have that 
      \[|q(x)| = (h(x)^{2}+ M^2)^{1/2},\qquad
      \arg q(x)=\pi+\arctan\left(\frac{M}{h(x)}\right).\]
      %
      %
      Therefore, for $M>0$ we have that 
      \begin{align*}
	 \re[ (h(x) + iM )^{1/2} ]
	 & = |h(x) + iM|^{1/2} \cos\left( \frac{1}{2} \arg( iM+h(x) ) \right)
     \\
	 & = |h(x) + iM|^{1/2} \cos\left( \frac{\pi}{2} + \frac{1}{2} \arctan \left( M/h(x) \right) \right)
	 \\
	 & 
	 \sim  -\frac{1}{2} |h(x)|^{1/2}  \arctan( M/h(x) )
	 \\
	 &\sim -\frac{M}{2} |h(x)|^{1/2} h(x)^{-1}
	 = \frac{M}{2} |h(x)|^{-1/2}
      \qquad\text{for } x\to\infty.
      \end{align*}
      Since $| (h(x) + iM)^{1/2}|^2 = (\re[ (h(x) + iM )^{1/2} ] )^2
	 + ( \im[ ( h(x) + iM )^{1/2} ] )^2 $,
      we obtain
      \begin{align*}
	 \im[ (h(x) + iM)^{1/2} ]
	 &\sim  |h(x)|^{1/2}
      \qquad\text{for } x\to\infty
      \end{align*}
      and consequently
      \begin{align*}
      \snull := \lim_{x\to\infty} \frac{ \re[(h(x) + iM)^{1/2}] }{|h(x) + iM|^{1/2}}
      = 0.
      \end{align*}
      The calculations for $M<0$ are analogous.
      Now \ref{thm:e4.3:infinity} and \ref{thm:e4.3:infinityLC} follow 
      from 
      Theorem~\ref{thm:criteriaReIm} \ref{criterion:LPIb} and \ref{criterion:LC}, respectively.
   \end{proof}
\end{exam}


\begin{exam}
We review an example of more relevance from theoretical physics. In \cite{bender} the authors consider for $N\geq 1$,
$N\in \mathbb N$,
   $$
   \textbf{q}(z) :=-(iz)^{N+2},
   \quad \mbox{with} \quad
   z\in \Gamma,
   $$
   where $\Gamma = \{ z\in\C : z = x e^{i\phi},\, x\in [1,\infty)\}$ is as in \eqref{Ozuna}. Then the function $q$ in \eqref{OzunaII} is given by
   $$
   q(x) 
   = -i^{N+2}x^{N+2}e^{(N+2)i\phi}
   = e^{i\pi + i(N+2)( \phi + \frac{\pi}{2} ) } x^{N+2}
   $$
   and the differential equation \eqref{pt-2} becomes
   \begin{equation}
       \label{5.1}
       v''(x) = e^{2i\phi}(e^{i\pi + i(N+2)( \phi + \frac{\pi}{2} ) } x^{N+2} - \lambda) v(x).
   \end{equation}
   We will show that \eqref{5.1} is in the limit point case~I if 
   \begin{equation}
       \label{Osterhase}
   \phi \notin \left\{\frac{(2k+1)\pi}{N+4} - \frac{\pi}{2} : k\in\{0, \dots, N+3\}\right\}
   \end{equation}
   and else in the limit circle case.
   This was already shown in \cite{LebTr} by a somewhat different argument.
   
   First note that 
   $$
   e^{2i\phi} q(x) 
   = e^{ i(N+4)( \phi + \frac{\pi}{2} ) } x^{N+2}.
   $$
   If we set 
   \begin{equation*}
       \psi_N := (N+4) (\phi + \textstyle\frac{\pi}{2} )
   \end{equation*}
   then
   \begin{align}
   \nonumber
   Q 
   =
   \cco\{ q(x) + r e^{-2i\phi} :\, r>0,\, x\ge 1\}
   &= e^{-2i\phi}\cco\{ e^{2i\phi}q(x) + r :\, 0 < r < \infty,\, x\in [1,\infty)\}
   \\
   \label{MajorTom}
   &= e^{-2i\phi}\cco\{ t e^{i\psi_N} + r :\, 0 < r < 
   \infty,\, t\in [1,\infty)\}.
   \end{align}
   Hence $Q$ is
   \begin{itemize}
       \item the halfline $\{ t e^{-2i\phi} : t\in [1,\infty)\}$ if  $\psi_N \in 2\pi \Z$, 
       that is, if 
       $\phi = \frac{2k\pi}{N+4} - \frac{\pi}{2}$ for some $k\in\Z$;
       
       \item the line $\{ t e^{-2i\phi} : t\in \R\}$ if  $\psi_N \in \pi + 2\pi \Z$, 
       that is, if 
       $\phi = \frac{(2k+1)\pi}{N+4} - \frac{\pi}{2}$ for some $k\in\Z$;

       \item a sector with centre at $e^{-2i\phi + i\psi_N}$ and angle smaller than $\pi$ in all other cases.
   \end{itemize}

   Clearly, $q$ satisfies (i) in Theorem~\ref{t3.1}
   and for any $\lambda\in \C\setminus Q$ conditions (ii) and (iii) of that theorem are satisfied, too, because for large enough $x$
   \begin{align*}
       \left| 
       \frac{ 5 (e^{2i\phi} q'(x))^2 }{ 16( e^{2i\phi} (q(x) - \lambda) )^{5/2}}
       \right|
       \le  x^{ - \frac{N}{2}-3 }
       \quad \mbox{and} \quad
       \left| 
       \frac{ e^{2i\phi} q''(x) }{ 4( e^{2i\phi} (q(x) - \lambda) )^{3/2}}
       \right|
       \le x^{-N/2 -3},
   \end{align*}
   which shows the integrability condition in (iii).
   As in \eqref{Kiew} we set
   $$
   s_\lambda(x):= e^{2i\phi} (q(x) - \lambda)
   = e^{i\psi_N } x^{N+2} - e^{2i\phi}\lambda.
   $$
   It is easy to see that the limit in \eqref{eq:defs0} exists and is 
   \begin{align*}
   \snull 
   = \lim_{x\to\infty} \frac{ \re[ s_\lambda(x)^{1/2}] }{ |s_\lambda(x)|^{1/2} }
   = \re e^{ \frac{i\psi_N}{2} }.
   \end{align*}
   
   Therefore, $\sigma \neq 0$ if and only if $\psi_N \notin \pi + 2\pi\Z$, that is, if $\phi$ satisfies \eqref{Osterhase}.
   For those $\phi$, by Theorem~\ref{thm:criteriaReIm} \ref{criterion:LPIa},
   the differential equation is in the limit point case~I.
   If on the other hand
    \begin{equation*}
       \label{Osterhase2}
   \phi =\frac{(2k+1)\pi}{N+4} - \frac{\pi}{2} \quad 
   \mbox{for some} \quad    k=0, \dots, N+3,
   \end{equation*}
    then 
   $s_\lambda(x)= -x^{N+2} - e^{2i\phi}\lambda$
   and the differential equation \eqref{5.1} becomes 
   \begin{equation*}
       v''(x) = (-x^{N+2} - e^{2i\phi}\lambda) v(x)
   \end{equation*}
   which is of the type discussed in Example~\ref{e4.3} \ref{thm:e4.3:infinityLC} with $h(x) = -x^{N+2} - \re(e^{2i\phi}\lambda)$.
   Therefore, \eqref{5.1} is in the limit circle case.
\end{exam}

\begin{center}
{\bf Acknowledgements.}
\end{center}
The authors are grateful for the financial support by 
the German Academic Exchange Service 
(Programm des Projektbezogenen Personenaustausches Kolumbien 2023-2025), Minciencias Colombia (Movilidad acad\'emica con Europa, 836-2019 and 923-2022)  and the Faculty of Sciences of University of the Andes (Proyecto semilla INV-2021-126-2291).


\end{document}